\documentclass[amstex,12pt,russian,amssymb]{article}

\usepackage{mathtext}
\usepackage[cp1251]{inputenc}
\usepackage[T2A]{fontenc}
\usepackage[russian]{babel}
\usepackage[dvips]{graphicx}
\usepackage{amsmath}
\usepackage{amssymb}
\usepackage{amsxtra}
\usepackage{latexsym}
\usepackage{ifthen}

\begin{document}

\newcounter{lemma}
\newcommand{\lemma}{\par \refstepcounter{lemma}%
{\bf Лемма \arabic{lemma}.}}

\newcounter{corollary}
\newcommand{\corollary}{\par \refstepcounter{corollary}%
{\bf Следствие \arabic{corollary}.}}

\newcounter{remark}
\newcommand{\remark}{\par \refstepcounter{remark}%
{\bf Замечание \arabic{remark}.}}

\newcounter{theorem}
\newcommand{\theorem}{\par \refstepcounter{theorem}%
{\bf Теорема \arabic{theorem}.}}

\newcounter{proposition}
\newcommand{\proposition}{\par \refstepcounter{proposition}%
{\bf Предложение \arabic{proposition}.}}

\renewcommand{\refname}{\centerline{\bf Список литературы}}

\newcommand{\proof}{{\it Доказательство.\,\,}}

\noindent УДК 517.5

{\bf Д.П.~Ильютко} (МГУ имени М.\,В.\,Ломоносова),

{\bf Е.А.~Севостьянов} (Житомирский государственный университет им.\
И.~Франко)

\medskip
{\bf Д.П.~Ільютко} (МДУ імені М.\,В.\,Ломоносова),

{\bf Є.О.~Севостьянов} (Житомирський державний університет ім.\
І.~Франко)

\medskip
{\bf D.P.~Ilyutko} (M.\,V.\,Lomonosov Moscow State University),

{\bf E.A.~Sevost'yanov} (Zhitomir State University of I.~Franko)

{\bf О локальных свойствах одного класса отображений на римановых
многообразиях}

{\bf Про локальні властивості одного класу відображень на ріманових
многовидах}

{\bf On local properties of one class of mappings on Riemannian
manifolds}

Настоящая работа посвящена изучению вопросов, находящихся на стыке
теории пространственных квазиконформных отображений и теории
римановых поверхностей. Получены теоремы о локальном поведении
одного класса открытых дискретных отображений с неограниченной
характеристикой квазиконформности между произвольными римановыми
многообразиями.

Дана робота присвячена вивченню питань, що знаходяться на стику
теорії просторових квазіконформних відображень та теорії ріманових
поверхонь. Отримано теореми про локальну поведінку одного класу
відкритих дискретних відображень з необмеженою характеристикою
квазіконформності на довільних ріманових многовидах.

The present paper is devoted to questions located at the junction of
the theory of space quasiconformal mappings and Riemannian surfaces.
Theorems on local behavior of one class of open discrete mappings
with unbounded characteristic of qua\-si\-con\-for\-ma\-li\-ty on
arbitrary Riemannian manifolds are obtained.

\newpage
{\bf 1. Введение.} Настоящая заметка посвящена изучению локальных
свойств одного класса отображений на римановых многообразиях,
которые предполагаются ниже только открытыми и дискретными
(инъективность отображений не обязательна). Исследования,
проведённые ниже, относятся, с одной стороны, к исследованию
отображений геометрическим методом (методом модулей), с другой
стороны -- к теории многообразий. Как известно, в последнее время
активно развивается теория отображений на римановых многообразиях и
общих метрических пространствах, в частности --- теория
квазиконформных отображений и отображений с конечным искажением,
см., напр., \cite{ARS}--\cite{Af$_1$}. Основы теории квазиконформных
отображений, как известно, заложены академиком М.\,А.~Лаврентьевым,
что в дальнейшем дало почву для зарождения теории отображений с
ограниченным искажением по Ю.\,Г.~Решетняку и теории отображений с
конечным искажением (см.~\cite{BGMV}, \cite{Cr$_2$}, \cite{IM} и
\cite{Re}--\cite{Va}).

\medskip
Среди указанных работ отдельно можно указать публикации, в которых
за основу определения изучаемого класса берутся верхние и нижние
оценки искажения конформного модуля семейств кривых (\cite{ARS},
\cite{Cr$_2$}, \cite{KRSS}--\cite{MRSY}, \cite{Af$_1$} и
\cite{Sev$_4$}). Следует отметить, что в ${\Bbb R}^n$ при
$n\geqslant 2$ подобные оценки установлены, в основном, для всех
известных классов отображений, таких как аналитические функции на
плоскости, квазиконформные отображения, а также пространственные
отображения с ограниченным и конечным искажением (см.,
напр.,~\cite{HP}, \cite{Ri} и \cite{Pol}).

\medskip
Перейдём к определениям и формулировкам основных результатов.
Следующие понятия могут быть найдены, напр., в \cite{Lee} и
\cite{PSh}. Напомним, что {\it $n$-мерным топологическим
многообразием} ${\Bbb M}^n$ называется хаусдорфово топологическое
пространство со счётной базой, каждая точка которого имеет
окрестность, гомеоморфную некоторому открытому множеству в ${\Bbb
R}^n.$ {\it Картой} на многообразии ${\Bbb M}^n$ будем называть пару
$(U, \varphi),$ где $U$ --- открытое множество пространства ${\Bbb
M}^n$ и $\varphi$ --- соответствующий гомеоморфизм множества $U$ на
открытое множество в ${\Bbb R}^n.$ Если $p\in U$ и
$\varphi(p)=(x^1,\ldots,x^n)\in {\Bbb R}^n,$ то соответствующие
числа $x^1,\ldots,x^n$ называются {\it локальными координатами
точки} $p.$ {\it Гладким многообразием} называется само множество
${\Bbb M}^n$ вместе с соответствующим набором карт $(U_{\alpha},
\varphi_{\alpha}),$ так, что объединение всех $U_{\alpha}$ по
параметру $\alpha$ даёт всё ${\Bbb M}^n$ и, кроме того, отображение,
осуществляющее переход от одной системы локальных координат к
другой, принадлежит классу $C^{\infty}.$

\medskip
Напомним, что {\it римановой метрикой} на гладком многообразии
${\Bbb M}^n$ называется положительно определённое гладкое
симметричное тензорное поле типа $(0,2).$ В частности, компоненты
римановой метрики $g_{kl}$ в различных локальных координатах $(U,
x)$ и $(V, y)$ взаимосвязаны посредством тензорного закона
$$'g_{ij}(x)=g_{kl}(y(x))\frac{\partial y^k}{\partial x^i}
\frac{\partial y^l}{\partial x^j}.$$
{\it Римановым многообразием} будем называть гладкое многообразие
вместе с римановой метрикой на нём. Длину гладкой кривой
$\gamma=\gamma(t),$ $t\in [t_1, t_2],$ соединяющей точки
$\gamma(t_1)=M_1\in {\Bbb M}^n,$ $\gamma(t_2)=M_2\in {\Bbb M}^n$ и
$n$-мерный объём ({\it меру объёма $v$}) множества $A$ на римановом
многообразии определим согласно соотношениям
\begin{equation}\label{eq8}
l(\gamma):=\int\limits_{t_1}^{t_2}\sqrt{g_{ij}(x(t))\frac{dx^i}{dt}\frac{dx^j}{dt}}\,dt,\quad
v(A)=\int\limits_{A}\sqrt{\det g_{ij}}\,dx^1\ldots dx^n\,.
\end{equation}
Ввиду положительной определённости тензора $g=g_{ij}(x)$ имеем:
$\det g_{ij}>0.$ {\it Геодезическим расстоянием} между точками $p_1$
и $p_2\in {\Bbb M}^n$ будем называть наименьшую длину всех
кусочно-гладких кривых в ${\Bbb M}^n,$ соединяющих точки $p_1$ и
$p_2.$ Геодезическое расстояние между точками $p_1$ и $p_2$ будем
обозначать символом $d(p_1, p_2)$ (всюду далее $d$ обозначает
геодезическое расстояние, если не оговорено противное). Так как
риманово многообразие, вообще говоря, не предполагается связным,
расстояние между любыми точками многообразия, вообще говоря, может
быть не определено. Хорошо известно, что любая точка $p$ риманова
многообразия ${\Bbb M}^n$ имеет окрестность $U\ni p$ (называемую
далее {\it нормальной окрестностью точки $p$}) и соответствующее
координатное отображение $\varphi\colon U\rightarrow {\Bbb R}^n,$
так, что геодезические сферы с центром в точке $p$ и радиуса $r,$
лежащие в окрестности $U,$ переходят при отображении $\varphi$ в
евклидовы сферы того же радиуса, а пучок геодезических кривых,
исходящих из точки $p,$ переходит в пучок радиальных отрезков в
${\Bbb R}^n$ (см.~\cite[леммы~5.9 и 6.11]{Lee}, см.\ также
комментарии на стр.~77 здесь же). Локальные координаты
$\varphi(p)=(x^1,\ldots, x^n)$ в этом случае называются {\it
нормальными координатами} точки $p.$ Стоит отметить, что в случае
связного многообразия ${\Bbb M}^n$ открытые множества метрического
пространства $({\Bbb M}^n, d)$ порождают топологию исходного
топологического пространства ${\Bbb M}^n$
(см.~\cite[лемма~6.2]{Lee}). Заметим, что в нормальных координатах
всегда тензорная матрица $g_{ij}(x)$ в точке $p$ --- единичная (а в
силу непрерывности $g$ в точках, близких к $p,$ эта матрица сколь
угодно близка к единичной; см.~\cite[пункт~(c)
предложения~5.11]{Lee}).

\medskip
Пусть $X$ и $Y$ --- два топологических пространства. Отображение
$f\colon X\rightarrow Y$ называется {\it открытым}, если $f(A)$
открыто в $Y$ для любого открытого $A\subset X,$ и {\it дискретным},
если для каждого $y\in Y$ все точки множества $f^{\,-1}(y)$ имеют
попарно непересекающиеся окрестности.

\medskip
Пусть $\left(X,\,d\right)$ и
$\left(X^{\,{\prime}},{d}^{\,{\prime}}\right)$ --- метрические
пространства с расстояниями  $d$  и ${d}^{\,{\prime}}$
соответственно. Семейство $\frak{F}$ непрерывных отображений
$f\colon X\rightarrow {X}^{\,\,\prime}$ называется {\it нормальным},
если из любой последовательности отображений $f_{m} \in \frak{F}$
можно выделить подпоследовательность $f_{m_{k}}$, которая сходится
локально равномерно в $X$ (т.е., равномерно на любых компактных
подмножествах $X$) к непрерывной функции $f\colon \,X\,\rightarrow\,
X^{\,\prime}.$

\medskip
Введенное  понятие  очень  тесно  связано  со  следующим. Семейство
$\frak{F}$ отображений $f\colon X\rightarrow {X}^{\,\prime}$
называется {\it равностепенно непрерывным в точке} $x_0 \in X,$ если
для любого $\varepsilon>0$ найдётся такое $\delta>0$, что
${d}^{\,\prime} \left(f(x),f(x_0)\right)<\varepsilon$ для всех таких
$x,$ что $d(x,x_0)<\delta$ и для всех $f\in \frak{F}.$ Говорят, что
$\frak{F}$ {\it равностепенно непрерывно}, если $\frak{F}$
равностепенно непрерывно в каждой  точке $x_0\in X.$ Согласно одной
из версий теоремы Арцела--Асколи (см., напр.,
\cite[пункт~20.4]{Va}), если $\left(X,\,d\right)$ --- сепарабельное
метрическое пространство, а $\left(X^{\,\prime},\,
d^{\,\prime}\right)$ --- компактное метрическое пространство, то
семейство $\frak{F}$ отображений $f\colon X\rightarrow
{X}^{\,\prime}$ нормально тогда  и только тогда, когда  $\frak{F}$
равностепенно непрерывно. Здесь и далее равностепенная непрерывность
семейства отображений $\{f\colon {\Bbb M}^n\rightarrow {\Bbb
M}_*^n\}$ понимается в смысле геодезических расстояний $d$ и
$d^{\,\prime}$ на римановых многообразиях ${\Bbb M}^n$ и ${\Bbb
M}_*^n,$ соответственно.

\medskip
Пусть $\left(X,\,d, \mu\right)$ --- произвольное метрическое
пространство, наделённое мерой $\mu$ и $G(x_0, r)=\{x\in X: d(x,
x_0)<r\}.$ Следующее определение может быть найдено, напр., в
\cite[разд.~4]{RSa}. Будем говорить, что интегрируемая в $G(x_0, r)$
функция ${\varphi}\colon D\rightarrow{\Bbb R}$ имеет {\it конечное
среднее колебание} в точке $x_0\in D$, пишем $\varphi\in FMO(x_0),$
если
%
%
%
%
$$\limsup\limits_{\varepsilon\rightarrow 0}\frac{1}{\mu(G(
x_0,\,\varepsilon))}\int\limits_{G(x_0,\,\varepsilon)}
|{\varphi}(x)-\overline{{\varphi}}_{\varepsilon}|\,
d\mu(x)<\infty,$$
%
%
где
$\overline{{\varphi}}_{\varepsilon}=\frac{1} {\mu(G(
x_0,\,\varepsilon))}\int\limits_{G( x_0,\,\varepsilon)}
{\varphi}(x)\, d\mu(x).$

\medskip
Всюду далее (если не оговорено противное) ${\Bbb M}^n$ и ${\Bbb
M}_*^n$ -- римановы многообразия с геодезическими расстояниями $d$ и
$d_*,$ соответственно. {\it Кривой} $\gamma$ мы называем непрерывное
отображение отрезка $[a,b]$ (открытого интервала $(a,b),$ либо
полуоткрытого интервала вида $[a,b)$ или $(a,b]$) в ${\Bbb M}^n,$
$\gamma\colon [a,b]\rightarrow {\Bbb M}^n.$ Под семейством кривых
$\Gamma$ подразумевается некоторый фиксированный набор кривых
$\gamma,$ а, если $f\colon{\Bbb M}^n\rightarrow {\Bbb M}_*^n$ ---
произвольное отображение, то
$f(\Gamma)=\left\{f\circ\gamma|\gamma\in\Gamma\right\}.$ Длину
произвольной кривой $\gamma\colon [a, b]\rightarrow {\Bbb M}^n,$
лежащей на многообразии ${\Bbb M}^n,$ можно определить как точную
верхнюю грань сумм $\sum\limits_{i=1}^{n-1} d(\gamma(t_i),
\gamma(t_{i+1}))$ по всевозможным разбиениям $a\leqslant
t_1\leqslant\ldots\leqslant t_n\leqslant b.$ Следующие определения в
случае пространства ${\Bbb R}^n$ могут быть найдены, напр., в
\cite[разд.~1--6, гл.~I]{Va}, см.\ также \cite[гл.~I]{Fu}. Борелева
функция $\rho\colon {\Bbb M}^n\,\rightarrow [0,\infty]$ называется
{\it допустимой} для семейства $\Gamma$ кривых $\gamma$ в ${\Bbb
M}^n,$ если линейный интеграл по натуральному параметру $s$ каждой
(локально спрямляемой) кривой $\gamma\in \Gamma$ от функции $\rho$
удовлетворяет условию $\int\limits_{0}^{l(\gamma)}\rho
(\gamma(s))ds\geqslant 1.$ В этом случае мы пишем:
$\rho\in\mathrm{adm}\,\Gamma.$ {\it Мо\-ду\-лем} семейства кривых
$\Gamma $ называется величина
$$M(\Gamma)=\inf\limits_{\rho\in\mathrm{adm}\,\Gamma}
\int\limits_D \rho ^n (x)\,dv(x).$$
(Здесь и далее $v$ означает меру объёма, определённую в
(\ref{eq8})). При этом, если $\mathrm{adm}\,\Gamma=\varnothing,$ то
полагаем: $M(\Gamma)=\infty$ (см.~\cite[разд.~6 на с.~16]{Va} либо
\cite[с.~176]{Fu}). Свойства модуля в некоторой мере аналогичны
свойствам меры Лебега $m$ в ${\Bbb R}^n.$ Именно, модуль пустого
семейства кривых равен нулю, $M(\varnothing)=0,$ обладает свойством
монотонности относительно семейств кривых, %
$ \Gamma_1\subset\Gamma_2\Rightarrow M(\Gamma_1)\leqslant
M(\Gamma_2), $
а также свойством полуаддитивности:
$ M\left(\bigcup\limits_{i=1}^{\infty}\Gamma_i\right)\leqslant
\sum\limits_{i=1}^{\infty}M(\Gamma_i) $
(см.~\cite[теорема~6.2, гл.~I]{Va} в ${\Bbb R}^n$ либо
\cite[теорема~1]{Fu} в случае более общих пространств с мерами).
Говорят, что семейство кривых $\Gamma_1$ \index{минорирование}{\it
минорируется} семейством $\Gamma_2,$ пишем $\Gamma_1\,>\,\Gamma_2,$
если для каждой кривой $\gamma\,\in\,\Gamma_1$ существует подкривая,
которая принадлежит семейству $\Gamma_2.$
В этом случае,
\begin{equation}\label{eq32*A}
\Gamma_1
> \Gamma_2 \quad \Rightarrow \quad M(\Gamma_1)\leqslant M(\Gamma_2)
\end{equation} (см.~\cite[теорема~6.4, гл.~I]{Va} либо
\cite[свойство~(c)]{Fu} в случае более общих пространств с мерами).

\medskip{}
Следующее определение для случая ${\Bbb R}^n$ может быть найдено,
напр., в работе \cite{SS}. Пусть ${\Bbb M}^n$ и ${\Bbb M}_*^n$ ---
римановы многообразия{\em,} $n\geqslant 2,$ $D$ -- область в ${\Bbb
M}^n,$ $x_0\in D,$ $Q\colon D\rightarrow [0,\infty]$ --- измеримая
относительно меры объёма $v$ функция, и число $r_0>0$ таково, что
замкнутый шар $\overline{B(x_0, r_0)}$ лежит в некоторой нормальной
окрестности $U$ точки $x_0.$ Пусть также $0<r_1<r_2<r_0,$
$A=A(r_1,r_2, x_0)=\{x\in {\Bbb M}^n: r_1<d(x, x_0)<r_2\},$
$S_i=S(x_0,r_i),$ $i=1,2,$ --- геодезические сферы с центром в точке
$x_0$ и радиусов $r_1$ и $r_2,$ соответственно, а
$\Gamma\left(S_1,\,S_2,\,A\right)$ обозначает семейство всех кривых,
соединяющих $S_1$ и $S_2$ внутри области $A.$
%
Отображение $f\colon D\rightarrow {\Bbb M}_*^n$ условимся называть
{\it кольцевым $Q$-отображением в точке $x_0\,\in\,D,$} если
соотношение
%
$$M\left(f\left(\Gamma\left(S_1,\,S_2,\,A\right)\right)\right)\ \leqslant
\int\limits_{A} Q(x)\cdot \eta^n(d(x, x_0))\ dv(x)$$
выполнено в кольце $A$ для произвольных $r_1,r_2,$ указанных выше, и
для каждой измеримой функции $\eta \colon  (r_1,r_2)\rightarrow
[0,\infty ]\,$ такой, что
$\int\limits_{r_1}^{r_2}\eta(r)dr\geqslant 1.$ Отображения типа
кольцевых $Q$-отображений были предложены к изучению О.~Мартио и
изучались им совместно с В.~Рязановым, У.~Сребро и Э.~Якубовым,
см.~\cite{MRSY}, см.\ также \cite{BGMV}. В пространстве ${\Bbb R}^n$
кольцевые $Q$-отображения являются обобщением аналитических функций
на плоскости ($Q\equiv 1$), квазиконформных отображений ($Q\leqslant
K=\mathrm{const}$) и отображений с ограниченным искажением
($Q\leqslant K=\mathrm{const}$) (см.~\cite[теорема~1,
лемма~6]{Pol}). Кроме того, вторым автором данной статьи было
показано, что в области $D\subset {\Bbb R}^n$ произвольные открытые
дискретные отображения $f\colon D\rightarrow {\Bbb R}^n$ класса
$W_{loc}^{1, n},$ имеющие меру множества точек ветвления, равную
нулю, и так называемая внешняя дилатация $K_O(x, f)$ которых
локально суммируема в степени $n-1,$ являются кольцевыми
$Q$-отображениями при $Q=K_O^{n-1}(x, f)$
(\cite[следствие~2]{Sev$_4$}). Для гомеоморфизмов класса
$W_{loc}^{1, n}$ таких, что $f^{\,-1}\in W_{loc}^{1, n}$ этот
результат установлен О.~Мартио, В.~Рязановым, У.~Сребро и
Э.~Якубовым (\cite[теоремы 8.1 и 8.6]{MRSY}) и Е.~Афанасьевой
(Е.~Смоловой) на римановых многообразиях
(см.~\cite[лемма~6]{Af$_1$}).

\medskip{}
Пусть $(X, d, \mu)$ --- метрическое пространство с метрикой $d,$
наделённое локально конечной борелевской мерой $\mu.$ Следуя
\cite[раздел 7.22]{He} будем говорить, что борелева функция
$\rho\colon  X\rightarrow [0, \infty]$ является {\it верхним
градиентом} функции $u\colon X\rightarrow {\Bbb R},$ если для всех
спрямляемых кривых $\gamma,$ соединяющих точки $x$ и $y\in X$
выполняется неравенство $|u(x)-u(y)|\leqslant
\int\limits_{\gamma}\rho\,ds,$ где, как обычно,
$\int\limits_{\gamma}\rho\,ds$ обозначает линейный интеграл от
функции $\rho$ по кривой $\gamma.$ Будем также говорить, что в
указанном пространстве $X$ выполняется $(1; p)$-неравенство
Пуанкаре, если найдутся постоянные $C\geqslant 1$ и $\tau>0$ так,
что для каждого шара $B\subset X,$ произвольной локально
ограниченной непрерывной функции $u\colon X\rightarrow {\Bbb R}$ и
любого её верхнего градиента $\rho$ выполняется следующее
неравенство:
$$\frac{1}{\mu(B)}\int\limits_{B}|u-u_B|d\mu(x)\leqslant C\cdot({\rm diam\,}B)\left(\frac{1}{\mu(\tau B)}
\int\limits_{\tau B}\rho^p d\mu(x)\right)^{1/p}\,,$$
где $u_B:=\frac{1}{\mu(B)}\int\limits_{B}u d\mu(x).$
Метрическое пространство $(X, d, \mu)$ назовём {\it
$\widetilde{Q}$-регулярным по Альфорсу} при некотором
$\widetilde{Q}\geqslant 1,$ если при каждом $x_0\in X,$ некоторой
постоянной $C\geqslant 1$ и произвольного $R<{\rm diam}\,X$
$$\frac{1}{C}R^{\widetilde{Q}}\leqslant \mu(B(x_0, R))\leqslant CR^{\widetilde{Q}}.$$
Заметим, что локально римановы многообразия являются $n$-регулярными
по Альфорсу (см.~\cite[лемма~5.1]{ARS}). Следует также заметить, что
если риманово многообразие $\widetilde{Q}$-регулярно по Альфорсу, то
$\widetilde{Q}=n$ (см.\ рассуждения на с.~61 в \cite{He} о
совпадении $\widetilde{Q}$ с хаусдорфовой размерностью пространства
$X,$ а также \cite[лемма~5.1]{ARS} о совпадении топологической и
хаусдорфовых размерностей областей риманового многообразия).
Справедлива следующая

\medskip
\begin{theorem}\label{theor4*!} {\sl\,
Пусть ${\Bbb M}^n$ и ${\Bbb M}_*^n$ --- римановы многообразия{\em,}
$n\geqslant 2,$ многообразие ${\Bbb M}_*^n$ связно{\em,} является
$n$-регулярным по Альфорсу{\em,} кроме того{\em,} в ${\Bbb M}_*^n$
выполнено $(1;n)$-неравенство Пуанкаре. Пусть $B_R\subset {\Bbb
M}_*^n$ --- некоторый фиксированный шар радиуса $R,$ $K$ --
фиксированный невырожденный континуум в $B_R,$ $D$
--- область в ${\Bbb M}^n$ и $Q\colon D\rightarrow [0, \infty]$ ---
функция{\em,} измеримая относительно меры объёма $v.$ Обозначим
через $\frak{R}_{x_0, Q, B_R, K}(D)$ семейство открытых дискретных
кольцевых $Q$-отображений $f\colon D\rightarrow B_R\setminus K$ в
точке $x_0\in D.$ Тогда семейство отображений $\frak{R}_{x_0, Q,
B_R, K}(D)$ является равностепенно непрерывным в точке $x_0\in D,$
если $Q\in FMO(x_0).$}
\end{theorem}

\medskip
Ввиду теоремы Арцела--Асколи (см., напр., \cite[пункт~20.4]{Va})
имеем также следующее

\medskip
\begin{corollary}\label{cor1} {\sl\, Предположим, что в условиях
теоремы~{\ref{theor4*!}} многообразие ${\Bbb M}_*^n$ является
компактным. Если $\frak{R}_{D, Q, B_R, K}(D)$ --- семейство{\em,}
состоящее из всех открытых дискретных отображений $f\colon
D\rightarrow B_R,$ являющихся кольцевыми $Q$-отоб\-ра\-же\-ни\-я\-ми
в каждой точке $x_0\in D$ и{\em,} кроме того{\em,} $Q\in FMO(x_0)$ в
каждой точке $x_0\in D,$ то класс $\frak{R}_{D, Q, B_R, K}(D)$
образует нормальное семейство отображений. }
\end{corollary}

\medskip
{\bf 2. Вспомогательные леммы}. Всюду далее ${\Bbb M}^n$ ---
риманово многообразие при $n\geqslant 2,$ $d$ --- геодезическое
расстояние на ${\Bbb M}^n,$
 \begin{gather*}
B(x_0, r)=\left\{x\in{\Bbb M}^n: d(x, x_0)<r\right\},\quad S(x_0,r)
= \{ x\,\in\,{\Bbb M}^n: d(x, x_0)=r\},
\end{gather*}
${\rm diam}\,A$ --- геодезический диаметр множества $A\subset {\Bbb
M}^n.$ Всюду далее граница $\partial D$ области $D\subset {\Bbb
M}^n$ и замыкание $\overline{D}$ области $D$ понимаются в смысле
геодезического расстояния $d.$ Перед тем, как мы приступим к
изложению вспомогательных результатов и основной части данного
раздела, дадим ещё одно важное определение (см.~\cite[раздел~3,
гл.~II]{Ri}). Пусть $D$ --- область риманового многообразия ${\Bbb
M}^n,$ $n\geqslant 2,$ $f\colon D \rightarrow {\Bbb M}_*^n$ ---
отображение, $\beta\colon [a,\,b)\rightarrow {\Bbb M}_*^n$ ---
некоторая кривая и $x\in\,f^{\,-1}\left(\beta(a)\right).$ Кривая
$\alpha\colon [a,\,c)\rightarrow D,$ $c\leqslant b,$ называется {\it
максимальным поднятием} кривой $\beta$ при отображении $f$ с началом
в точке $x,$ если $(1)\quad \alpha(a)=x;$ $(2)\quad
f\circ\alpha=\beta|_{[a,\,c)};$ $(3)$\quad если
$c<c^{\prime}\leqslant b,$ то не существует кривой
$\alpha^{\prime}\colon [a,\,c^{\prime})\rightarrow D,$ такой что
$\alpha=\alpha^{\prime}|_{[a,\,c)}$ и $f\circ
\alpha=\beta|_{[a,\,c^{\prime})}.$ Имеет место следующее

\medskip
\begin{proposition}\label{pr7}
{\sl Пусть ${\Bbb M}^n$ и ${\Bbb M}_*^n$ --- римановы
многообразия{\em,} $n\geqslant 2,$ $D$ --- область в ${\Bbb M}^n,$
$f\colon D\rightarrow {\Bbb M}^n_*$ --- открытое дискретное
отображение{\em,} $\beta\colon [a,\,b)\rightarrow {\Bbb M}_*^n$ ---
кривая и точка $x\in\,f^{-1}\left(\beta(a)\right).$ Тогда кривая
$\beta$ имеет максимальное поднятие при отображении $f$ с началом в
точке $x.$ }
\end{proposition}

\medskip
\begin{proof} Зафиксируем точку $x_0\in {\Bbb M}^n$ и рассмотрим $f(x_0)\in {\Bbb
M}_*^n.$ Поскольку точка $f(x_0)$ принадлежит многообразию ${\Bbb
M}_*^n,$ найдётся окрестность $V$ этой точки, гомеоморфная множеству
$\psi(V)\subset {\Bbb R}^n.$ В силу непрерывности отображения $f,$
найдётся окрестность $U$ точки $x_0,$ такая что $f(U)\subset V.$ С
другой стороны, не ограничивая общности, можно считать, что $U$
гомеоморфна открытому множеству $\varphi(U)$ в ${\Bbb R}^n.$ Можно
также считать, что $\varphi(U)$ и $\psi(V)$ являются областями в
${\Bbb R}^n,$ тогда $f^*=\psi\circ f\circ\varphi^{\,-1}$ ---
открытое дискретное отображение между областями $\varphi(U)$ и
$\psi(V)$ в ${\Bbb R}^n.$ Для таких отображений существование
максимальных поднятий локально вытекает из соответствующего
результата Рикмана в $n$-мерном евклидовом пространстве
(см.~\cite[шаг~2 доказательства теоремы~3.2 гл.~II]{Ri}). Отсюда
вытекает локальное существование максимальных поднятий и на
многообразиях. Глобальное существование максимальных поднятий может
быть установлено аналогично доказательству шага 1 указанной выше
теоремы.~$\Box$
\end{proof}

Напомним определения классов Соболева и отображений с конечным
искажением, которые будут необходимы нам в дальнейшем
(см.~\cite{ARS}, \cite{IM}  и \cite{Maz}).

Пусть $U$ --- открытое множество, $U\subset {\Bbb R}^n,$ $u\colon
U\rightarrow {\Bbb R}$ --- некоторая функция, $u\in
L_{loc}^{\,1}(U).$ Предположим, что найдётся функция $v\in
L_{loc}^{\,1}(U),$ такая что
$$\int\limits_U \frac{\partial \varphi}{\partial x_i}(x)u(x)dm(x)=
-\int\limits_U \varphi(x)v(x)dm(x)$$
для любой функции $\varphi\in C_1^{\,0}(U).$ Тогда будем говорить,
что функция $v$ является {\it обобщённой производной первого порядка
функции $u$ по переменной $x_i,$} и обозначать символом:
$\frac{\partial u}{\partial x_i}(x):=v.$

Функция $u\in W_{loc}^{1,1}(U),$ если $u$ имеет обобщённые
производные первого порядка по каждой из переменных в $U,$ которые
являются локально интегрируемыми в $U.$

Пусть $G$ --- открытое множество в ${\Bbb R}^n.$ Отображение
$f\colon D\rightarrow {\Bbb R}^n$ принадлежит {\it классу Соболева}
$W^{1,1}_{loc}(G),$ пишут $f\in W^{1,1}_{loc}(G),$ если все
координатные функции $f=(f_1,\ldots,f_n)$ обладают обобщёнными
частными производными первого порядка, которые локально интегрируемы
в $G$ в первой степени.

\medskip
Пусть теперь ${\Bbb M}^n$ и ${\Bbb M}_*^n$ --- римановы многообразия
и область $D\subset {\Bbb M}^n,$ тогда будем говорить, что
отображение $f\colon D\rightarrow {\Bbb M}^n_*$ принадлежит классу
$f\in W^{1, 1}_{loc}(D),$ если каждая пара точек $p\in D$ и $f(p)\in
f(D)$ имеют окрестности $U\subset D,$ $V\subset f(D),$ в которых
$\psi\circ f\circ \varphi^{-1}\in W^{1,1}_{loc}(\varphi(U)),$ где
$\varphi\colon U\rightarrow {\Bbb R}^n$ и $\psi\colon V\rightarrow
{\Bbb R}^n$ --- соответствующие карты, переводящие $U$ и $V$ в
некоторые открытые подмножества ${\Bbb R}^n.$

\medskip{}
Пусть $A$ --- открытое подмножество многообразия ${\Bbb M}^n,$ а $C$
--- компактное подмножество $A.$  {\it Конденсатором} будем называть
пару множеств $E=\left(A,\,C\right).$ {\it Ёмкостью} конденсатора
$E$ называется следующая величина:
\begin{equation}\label{eq1.1AB} {\rm cap}\,E={\rm
cap}\,\left(A,\,C\right)= \inf\limits_{u\in W_0(E)}\,\,\int\limits_A
|\nabla u(x)|^n\,\,dv(x),
\end{equation}
где $W_0(E)=W_0\left(A,\,C\right)$ --- семейство неотрицательных
непрерывных функций $u\colon A\rightarrow{\Bbb R}$ с компактным
носителем в $A,$ таких что $u(x)\geqslant 1$ при $x\in C$ и $u\in
W_{loc}^{1, 1}.$
В формуле~\eqref{eq1.1AB}, как обычно, $|\nabla
u|={\left(\sum\limits_{i=1}^n\,{\left(\partial_i u\right)}^2
\right)}^{1/2}.$

Следующее утверждение имеет важное значение для доказательства
дальнейших результатов (см.~\cite[предложение~10.2 и замечание~10.8,
гл.~II]{Ri}).

\medskip
 \begin{proposition}\label{pr1*!}
{\,\sl Пусть ${\Bbb M}^n$ --- риманово многообразие{\em,}
$n\geqslant 2,$ $E=(A,\,C)$ --- произвольный конденсатор в ${\Bbb
M}^n$ и пусть $\Gamma_E$ --- семейство всех кривых вида
$\gamma\colon [a,\,b)\rightarrow A,$ таких что $\gamma(a)\in C$ и
$|\gamma|\cap\left(A\setminus F\right)\ne\varnothing$ для
произвольного компакта $F\subset A.$
Тогда
$ {\rm cap}\,E=M(\Gamma_E). $
}
\end{proposition}

\medskip{}
Следующая лемма может быть полезной при исследовании свойства
равностепенной непрерывности открытых дискретных кольцевых
$Q$-отоб\-ра\-же\-ний в наиболее общей ситуации. Её доказательство
аналогично случаю ${\Bbb R}^n$ (см. \cite[лемма~4.1]{SS}), однако,
для полноты и строгости изложения мы приводим его полностью для
случая произвольных римановых многообразий.

\medskip
\begin{lemma}\label{lem4}{\sl\,
Пусть ${\Bbb M}^n$ и ${\Bbb M}_*^n$ --- римановы многообразия{\em,}
$n\geqslant 2,$ $D$ --- область в ${\Bbb M}^n,$ $f\colon
D\rightarrow{\Bbb M}_*^n$ --- открытое дискретное кольцевое
$Q$-отоб\-ра\-же\-ние в точке $x_0\in D,$ $r_0>0$ таково{\em,} что
шар $B(x_0, r_0)$ лежит со своим замыканием в некоторой нормальной
окрестности $U$ точки $x_0.$ Предположим{\em,} что для некоторого
числа $0<\varepsilon_0<r_0,$ некоторого $\varepsilon_0^{\,\prime}\in
(0, \varepsilon_0)$ и семейства неотрицательных измеримых по Лебегу
функций $\{\psi_{\varepsilon}(t)\},$ $\psi_{\varepsilon}\colon
(\varepsilon, \varepsilon_0)\rightarrow [0, \infty],$
$\varepsilon\in\left(0, \varepsilon_0^{\,\prime}\right),$ выполнено
условие
 \begin{equation} \label{eq3.7B}
\int\limits_{\varepsilon<d(x, x_0)<\varepsilon_0}
Q(x)\cdot\psi_{\varepsilon}^n(d(x, x_0))\, dv(x)\leqslant
F(\varepsilon, \varepsilon_0)\qquad\forall\,\,\varepsilon\in(0,
\varepsilon_0^{\,\prime}),
 \end{equation}
где $F(\varepsilon, \varepsilon_0)$ --- некоторая функция и
\begin{equation}\label{eq3AB} 0<I(\varepsilon, \varepsilon_0):=
\int\limits_{\varepsilon}^{\varepsilon_0}\psi_{\varepsilon}(t)dt <
\infty\qquad\forall\,\,\varepsilon\in(0,
\varepsilon_0^{\,\prime}).\end{equation}
Тогда
\begin{equation}\label{eq3B}
{\rm cap}\,f(E)\leqslant F(\varepsilon,\varepsilon_0)/
I^{n}(\varepsilon, \varepsilon_0)\qquad
\forall\,\,\varepsilon\in\left(0,\,\varepsilon_0^{\,\prime}\right),
\end{equation}
где $E=\left(A,\,C\right)$ --- конденсатор{\em,} $A=B(x_0,
\varepsilon_0)$ и $C=\overline{B(x_0, \varepsilon)}.$
}
\end{lemma}

\medskip
\begin{proof}
Рассмотрим конденсатор $E=(A,\,C),$ где $A$ и $C$ таковы, как
указано в условии леммы. Если ${\rm cap}\,f(E)=0,$ доказывать
нечего. Пусть ${\rm cap}\,f(E)\ne 0.$

\medskip
Пусть $\Gamma_E$ --- семейство всех кривых вида $\gamma\colon
[a,\,b)\rightarrow A,$ таких что $\gamma(a)\in C$ и
$|\gamma|\cap\left(A\setminus F\right)\ne\varnothing$ для
произвольного компакта $F\subset A,$ где $|\gamma|=\{x\in {\Bbb
M}^n: \exists\,t\in [a, b): \gamma(t)=x \}$ --- носитель кривой
$\gamma.$ Напомним, что ${\rm cap}\,E=M(\Gamma_E)$ (см.\ предложение
\ref{pr1*!}). Для конденсатора $f(E)$ рассмотрим семейство кривых
$\Gamma_{f(E)}.$ Заметим также, что каждая кривая
$\gamma\in\Gamma_{f(E)}$ имеет максимальное поднятие при отображении
$f,$ лежащее в $A$ с началом в $C$ (см.\ предложение~\ref{pr7}).
Пусть $\Gamma^{\,*}$ --- семейство всех максимальных под\-ня\-тий
кривых $\Gamma_{f(E)}$ при отображении $f$ с началом в $C.$ Покажем,
что $\Gamma^{*}\subset \Gamma_E.$

\medskip
Предположим противное, т.е., что существует кривая $\beta\colon
[a,\,b)\rightarrow {\Bbb R}^n$ семейства $\Gamma_{f(E)},$ для
которой соответствующее максимальное поднятие $\alpha\colon
[a,\,c)\rightarrow B(x_0, \varepsilon_0)$ лежит в некотором компакте
$K$ внутри $B(x_0, \varepsilon_0).$ Следовательно, его замыкание
$\overline{\alpha}$ --- компакт в $B(x_0, \varepsilon_0).$ Заметим,
что $c\ne b,$ поскольку в противном случае $\overline{\beta}$ ---
ком\-пакт в $f\left(B(x_0, \varepsilon_0)\right),$ что противоречит
условию $\beta\,\in\,\Gamma_{f(E)}.$ Рассмотрим предельное множество
$G$ кривой $\alpha$ при $t_k\rightarrow c,$
$$G=\left\{x\in {\Bbb R}^n: x=\lim\limits_{k\rightarrow\,\infty} \alpha(t_k)
 \right\},\quad t_k\,\in\,[a,\,c),\quad
 \lim\limits_{k\rightarrow\infty}t_k=c.$$
Отметим, что переходя к подпоследовательностям, здесь можно
ограничиться монотонными последовательностями $t_k.$ Для $x\in G,$ в
силу непрерывности $f,$ будем иметь
$f\left(\alpha(t_k)\right)\rightarrow\,f(x)$ при
$k\rightarrow\infty,$ где $t_k\in[a,\,c),\,t_k\rightarrow c$ при
$k\rightarrow \infty.$ Однако,
$f\left(\alpha(t_k)\right)=\beta(t_k)\rightarrow\beta(c)$ при
$k\rightarrow\infty.$ Отсюда заключаем, что $f$ постоянна на $G$ в
$B(x_0, \varepsilon_0).$ С другой стороны, по условию Кантора в
компакте $\overline{\alpha}$  (см.~\cite[разд.~3.6, гл.~I]{Wh}),
$$G\,=\,\bigcap\limits_{k\,=\,1}^{\infty}\,\overline{\alpha\left(\left[t_k,\,c\right)\right)}=
\limsup\limits_{k\rightarrow\infty}\alpha\left(\left[t_k,\,c\right)\right)=
\liminf\limits_{k\rightarrow\infty}\alpha\left(\left[t_k,\,c\right)\right)\ne\varnothing
$$
%
в виду монотонности последовательности связных множеств
$\alpha\left(\left[t_k,\,c\right)\right),$ откуда следует, что $G$
является связным согласно \cite[разд.~9.12, гл.~I]{Wh}. Таким
образом, в силу дискретности $f,$ $G$ не может состоять более чем из
одной точки, и кривая $\alpha\colon [a,\,c)\rightarrow\,B(x_0,
\varepsilon_0)$ продолжается до замкнутой кривой $\alpha\colon
[a,\,c]\rightarrow K\subset B(x_0, \varepsilon_0),$ причём
$f\left(\alpha(c)\right)=\beta(c).$ Снова по предложению~\ref{pr7}
можно построить максимальное поднятие $\alpha^{\,\prime}$ кривой
$\beta|_{[c,\,b)}$ с началом в точке $\alpha(c).$ Объединяя поднятия
$\alpha$ и $\alpha^{\,\prime},$ получаем новое поднятие
$\alpha^{\,\prime\prime}$ кривой $\beta,$ которое определено на $[a,
c^{\prime}),$ \,\,$c^{\,\prime}\,\in\,(c,\,b),$ что противоречит
максимальности поднятия $\alpha.$ Таким образом,
$\Gamma^{\,*}\subset\Gamma_E.$

Заметим, что $\Gamma_{f(E)}>f(\Gamma^{*}),$ и, следовательно, ввиду
свойства~\eqref{eq32*A}
\begin{equation}\label{eq7AB}
M\left(\Gamma_{f(E)}\right)\leqslant M\left(f(\Gamma^{*})\right).
\end{equation}
Рассмотрим
$S_{\,\varepsilon}=S(x_0,\,\varepsilon) = \{ x\,\in\,{\Bbb M}^n :
d(x, x_0)=\,\varepsilon\},$ $
S_{\,\varepsilon_{0}}=S(x_0,\,\varepsilon_0) = \{ x\,\in\,{\Bbb M}^n
: d(x, x_0)=\,\varepsilon_0\},$
где $\varepsilon_0$ --- из условия леммы и
$\varepsilon\in\left(0,\,\varepsilon_0^{\,\prime}\right).$ Всюду
ниже, как и прежде, полагаем
$$A(r_1, r_2, x_0)=\{x\in {\Bbb M}^n: r_1<d(x, x_0)<r_2\}.$$ Заметим,
что, поскольку $\Gamma^{*}\,\subset\,\Gamma_E,$ то при всех
достаточно малых $\delta>0$ будем иметь, что
$\Gamma\left(S_{\varepsilon}, S_{\varepsilon_0-\delta},
A(\varepsilon, \varepsilon_0-\delta, x_0)\right)<\Gamma^{*}$ и,
следовательно, $f(\Gamma\left(S_{\varepsilon},
S_{\varepsilon_0-\delta}, A(\varepsilon, \varepsilon_0-\delta,
x_0)\right))<f(\Gamma^{\,*})$ (здесь мы положили
$S_{\varepsilon_0-\delta}:=\{x\in {\Bbb R}^n:
|x-x_0|<\varepsilon_0-\delta\}$). Значит,
\begin{equation}\label{eq3.3.1}
M\left(f(\Gamma^{*})\right)\leqslant
 M\left(f\left(\Gamma\left(S_{\varepsilon}, S_{\varepsilon_0-\delta},
 A(\varepsilon, \varepsilon_0-\delta, x_0)\right)\right)\right).
\end{equation}
Из соотношений~\eqref{eq7AB} и~\eqref{eq3.3.1} следует, что
$$
M\left(\Gamma_{f(E)}\right)\,\leqslant
 M\left(f\left(\Gamma\left(S_{\varepsilon}, S_{\varepsilon_0-\delta},
 A(\varepsilon, \varepsilon_0-\delta, x_0)\right)\right)\right)
$$
и, таким образом, по предложению \ref{pr1*!}
\begin{equation}\label{eq5aa}
{\rm cap}\,\,f(E) \leqslant
 M\left(f\left(\Gamma\left(S_{\varepsilon}, S_{\varepsilon_0-\delta},
 A(\varepsilon, \varepsilon_0-\delta, x_0)\right)\right)\right).
\end{equation}
Пусть $\eta(t)$ произвольная неотрицательная измеримая функция,
удовлетворяющая условию
$\int\limits_{\varepsilon}^{\varepsilon_0}\eta(t)dt=1.$ Рассмотрим
семейство измеримых функций
$\eta_{\delta}(t)=\frac{\eta(t)}{\int\limits_{\varepsilon}^{\varepsilon_0-\delta}\eta(t)dt}.$
(Так как $\int\limits_{\varepsilon}^{\varepsilon_0}\eta(t)dt=1,$ то
$\delta>0$ можно выбрать так, что
$\int\limits_{\varepsilon}^{\varepsilon_0-\delta}\eta(t)dt>0$).
Поскольку
$\int\limits_{\varepsilon}^{\varepsilon_0-\delta}\eta_{\delta}(t)dt=1,$
то по определению кольцевого $Q$-отображения в точке $x_0$ мы
получим
 \begin{multline*}
M\left(f\left(\Gamma\left(S_{\varepsilon}\,,
S_{\varepsilon_0-\delta}, A(\varepsilon, \varepsilon_0-\delta,
x_0)\right)\right)\right)\,\leqslant\\
\leqslant\frac{1}{\left(\int\limits_{\varepsilon}^{\varepsilon_0-\delta}\eta(t)dt\right)^n}
\int\limits_{\varepsilon<|x-x_0|<\varepsilon_0}Q(x)\cdot\eta^n(d(x,
x_0))\, \ dv(x).
 \end{multline*}
Переходя здесь к пределу при $\delta\rightarrow 0$ и учитывая
соотношение \eqref{eq5aa}, получаем что
\begin{equation}\label{eq1F}
{\rm cap}\,\,f(E) \leqslant
\int\limits_{\varepsilon<|x-x_0|<\varepsilon_0}Q(x)\cdot\eta^n(d(x,
x_0))\, \ dv(x)\end{equation}
для произвольной неотрицательной измеримой функции $\eta(t)$ такой
что $\int\limits_{\varepsilon}^{\varepsilon_0}\eta(t)dt=1.$
Рассмотрим семейство измеримых функций
$\eta_{\varepsilon}(t)=\psi_{\varepsilon}(t)/I(\varepsilon,
\varepsilon_0 ),$ $t\in(\varepsilon,\, \varepsilon_0).$ Заметим, что
для $\varepsilon\in (0, \varepsilon_0^{\,\prime})$
выполнено равенство
$\int\limits_{\varepsilon}^{\varepsilon_0}\eta_{\varepsilon}(t)\,dt=1.$
Тогда из (\ref{eq1F}) мы получим, что для любого $\varepsilon\in (0,
\varepsilon_0^{\,\prime})$
 \begin{equation}\label{eq8A}
{\rm cap}\,\,f(E)\leqslant\frac{1}{I^n(\varepsilon, \varepsilon_0)}
\int\limits_{\varepsilon<|x-x_0|<\varepsilon_0}Q(x)\cdot\psi_{\varepsilon}^n(d(x,
x_0))\,dv(x).
 \end{equation}
Из соотношений~\eqref{eq3.7B} и \eqref{eq8A} следует
соотношение~\eqref{eq3B}.~$\Box$
 \end{proof}

\medskip
Справедливо следующее утверждение (см.~\cite[предложение~4.7]{AS}).

\medskip
\begin{proposition}\label{pr2}
{\sl Пусть $X$ --- $\widetilde{Q}$-регулярное по Альфорсу
метрическое пространство с мерой{\em,} в котором выполняется
$(1;n)$-неравенство Пуанкаре так{\em,} что $\widetilde{Q}-1\leqslant
n\leqslant \widetilde{Q}.$ Тогда для произвольных континуумов $E$ и
$F,$ содержащихся в шаре $B(x_0, R),$ и некоторой постоянной $C>0$
выполняется неравенство
$$M(\Gamma(E, F, X))\geqslant \frac{1}{C}\cdot\frac{\min\{{\rm diam}\,E, {\rm diam}\,F\}}{R^{1+n-\widetilde{Q}}}.$$ }
\end{proposition}

Теперь сформулируем и докажем утверждение о равностепенной
непрерывности кольцевых $Q$-отображений между римановыми
многообразиями в <<максимальной>> степени общности.

\medskip
\begin{lemma}\label{lem1}{\sl\,
Пусть ${\Bbb M}^n$ и ${\Bbb M}_*^n$ --- римановы многообразия{\em,}
$n\geqslant 2,$ многообразие ${\Bbb M}_*^n$ связно{\em,} является
$n$-регулярным по Альфорсу{\em,} кроме того{\em,} в ${\Bbb M}_*^n$
выполнено $(1;n)$-неравенство Пуанкаре. Предположим{\em,} $D$ ---
область в ${\Bbb M}^n$ и $\frak{R}_{x_0, Q, B_R, K}(D)$ ---
семейство открытых дискретных кольцевых $Q$-отоб\-ра\-же\-ний
$f\colon D\rightarrow B_R\setminus K$ в точке $x_0\in D,$ где
$B_R\subset {\Bbb M}_*^n$
--- некоторый фиксированный шар радиуса $R,$ $K$ -- невырожденный континуум в $B_R.$
Пусть $\varepsilon_0<r_0,$ где шар $B(x_0, r_0)$ лежит
вместе со своим замыканием в некоторой нормальной окрестности $U$
точки $x_0.$

Предположим также{\em,} что для некоторого числа
$\varepsilon_0^{\,\prime}\in (0, \varepsilon_0)$ и семейства
неотрицательных измеримых по Лебегу функций
$\{\psi_{\varepsilon}(t)\},$ $\psi_{\varepsilon}\colon (\varepsilon,
\varepsilon_0)\rightarrow [0, \infty],$ $\varepsilon\in\left(0,
\varepsilon_0^{\,\prime}\right),$ выполнено
условие~\eqref{eq3.7B}{\em,} где некоторая заданная функция
$F(\varepsilon, \varepsilon_0)$ удовлетворяет условию
$F(\varepsilon, \varepsilon_0)=o(I^n(\varepsilon, \varepsilon_0)),$
а $I(\varepsilon, \varepsilon_0)$ определяется
соотношением~\eqref{eq3AB}.

Тогда семейство отображений $\frak{R}_{Q, x_0, B_R, K}(D)$ является
равностепенно непрерывным в точке $x_0.$ }
 \end{lemma}

\medskip
\begin{proof}
Пусть $f\in\frak{R}_{x_0, Q, B_R, K}(D)$. Полагаем $A:=B(x_0,
r_0)\subset D.$ Заметим, что при указанных условиях $\overline{A}$
является компактным подмножеством $D.$ Тогда при каждом
$0<\varepsilon<r_0$ множество $C:=\overline{B(x_0, \varepsilon)}$
является компактным подмножеством $B(x_0, r_0),$ поскольку $C$ есть
замкнутое подмножество компактного пространства $\overline{A}$
(см.~\cite[теорема~2 пункта II гл. 4]{Ku}). Таким образом, $E=(A,
C)$ --- конденсатор в ${\Bbb M}^n.$

Рассмотрим семейство кривых $\Gamma_{f(E)}$ для конденсатора $f(E)$
в терминах предложения $\ref{pr1*!}.$ Заметим, что подсемейство
неспрямляемых кривых семейства $\Gamma_{f(E)}$ имеет нулевой модуль,
и что оставшееся подсемейство, состоящее из всех спрямляемых кривых
семейства $\Gamma_{f(E)},$ состоит из кривых $\beta\colon [a,
b)\rightarrow f(D),$ имеющих предел при $t\rightarrow b.$ Заметим,
что указанный предел принадлежит множеству $\partial f(A).$ Из
сказанного следует, что
\begin{equation}\label{eq1}
M(\Gamma_{f(E)})=M(\Gamma(f(C), \partial f(A), f(A))).
\end{equation}
Заметим, что $\Gamma(K, f(C), {\Bbb M}_*^n)>\Gamma(f(C),
\partial f(A), f(A))$ (см. \cite[теорема
1, $\S\,46,$ п.~I]{Ku}), так что ввиду (\ref{eq32*A})
\begin{equation}\label{eq9}
M(\Gamma(f(C), \partial f(A), f(A)))\geqslant M(\Gamma(K, f(C),
{\Bbb M}_*^n))\,.
\end{equation}
Ввиду предложения \ref{pr2} получим:
\begin{equation}\label{eq2}
M(\Gamma(K, f(C), {\Bbb M}_*^n))\geqslant
\frac{1}{C_1}\cdot\frac{\min\{{\rm diam}\,f(C), {\rm
diam}\,K\}}{R}\,.
 \end{equation}
По лемме~\ref{lem4} $M(\Gamma_{f(E)})\rightarrow 0$ при
$\varepsilon\rightarrow 0,$ так что ввиду~\eqref{eq1} и~\eqref{eq2}
получаем, что
$$\min\{{\rm diam}\,f(C), {\rm diam}\,K\}={\rm diam}\,f(C)$$ при
$\varepsilon\rightarrow 0.$ Из соотношений~\eqref{eq3B}
и~\eqref{eq2} также вытекает, что для любого $\sigma>0$ найдётся
$\varepsilon_0=\varepsilon_0(\sigma)$ так, что при $\varepsilon\in
(0, \varepsilon_0)$
$${\rm diam}\,f(C)\leqslant \sigma,$$
что и означает равностепенную непрерывность семейства
$\frak{R}_{x_0, Q, B_R, K}(D)$ в точке $x_0.$~$\Box$
\end{proof}

\medskip{} Следующее утверждение вытекает из~\cite[лемма~4.1]{RSa} и \cite[следствие~5.1]{ARS}.

\medskip
\begin{proposition}\label{pr3}
{\sl В произвольном римановом многообразии ${\Bbb M}^n$ функция
$Q(x)$ класса $FMO$ в точке $x_0$ удовлетворяет
соотношению~\eqref{eq3.7B}{\em,} где функция
$G(\varepsilon):=F(\varepsilon, \varepsilon_0)/I^n(\varepsilon,
\varepsilon_0)$ удовлетворяет условию\/{\em:}
$G(\varepsilon)\rightarrow 0$ при $\varepsilon\rightarrow 0$ и
$\psi_{\varepsilon}(t)\equiv\psi(t):=\frac{1}{t\log\frac{1}{t}}.$}
\end{proposition}

\medskip
{\bf 3. Доказательство основных результатов.} {\it Доказательство
теоремы~{\em\ref{theor4*!}} вытекает из леммы~{\em\ref{lem1}} на
основании предложения~{\em\ref{pr3}}}.~$\Box$

\medskip{}
{\it Элементом площади} гладкой поверхности $H$ на
римановом многообразии ${\Bbb M}^n$ будем называть выражение вида
$$d\mathcal{A}=\sqrt{{\rm det}\,g_{\alpha\beta}^*}\,du^1\ldots du^{n-1},$$
где $g_{\alpha\beta}^*$ --- риманова метрика на $H,$ порождённая
исходной римановой метрикой $g_{ij}$ согласно соотношению
\begin{equation}\label{eq5}
g_{\alpha\beta}^*(u)=g_{ij}(x(u))\frac{\partial x^i}{\partial
u^{\alpha}} \frac{\partial x^j}{\partial u^{\beta}}.
\end{equation}
Здесь индексы $\alpha$ и $\beta$ меняются от $1$ до $n-1,$ а $x(u)$
обозначает параметризацию поверхности $H$ такую, что $\nabla_u x\ne
0.$ Справедливо следующее утверждение, обобщающее теорему
\ref{theor4*!}.

\medskip
\begin{theorem}\label{th1} {\sl\,
Пусть ${\Bbb M}^n$ и ${\Bbb M}_*^n$ --- римановы многообразия{\em,}
$n\geqslant 2,$ многообразие ${\Bbb M}_*^n$ связно{\em,} является
$n$-регулярным по Альфорсу{\em,} кроме того{\em,} в ${\Bbb M}_*^n$
выполнено $(1;n)$-неравенство Пуанкаре. Пусть $B_R\subset {\Bbb
M}_*^n$ --- некоторый фиксированный шар радиуса $R,$
$\overline{B_R}\ne{\Bbb M}_*^n,$ $D$ --- область в ${\Bbb M}^n$ и
$Q\colon D\rightarrow [1, \infty]$ --- функция{\em,} измеримая по
Лебегу. Обозначим через $\frak{R}_{x_0, Q, B_R, K}(D)$ семейство
открытых дискретных кольцевых $Q$-отображений $f\colon D\rightarrow
B_R\setminus K$ в точке $x_0\in D.$ Тогда семейство отображений
$\frak{R}_{x_0, Q, B_R, K}(D)$ является равностепенно непрерывным в
точке $x_0\in D,$ если при некотором $\delta(x_0)>0$ выполняется
равенство
\begin{equation}\label{eq3}
\int\limits_{0}^{\delta(x_0)}\frac{dt}{\left(\int\limits_{S(x_0,
t)}Q(x)\,d\mathcal{A}\right)^{\frac{1}{n-1}}}=\infty.
\end{equation}
 }
\end{theorem}

\begin{proof} Достаточно показать, что условие~\eqref{eq3} влечёт
выполнение условия~\eqref{eq3.7B} леммы~\ref{lem1}. Можно считать,
что $B(x_0, \delta(x_0))$ лежит в нормальной окрестности точки
$x_0.$ Рассмотрим функцию
$$
\psi(t)\quad =\quad\left\{
\begin{array}{rr}
\left(\int\limits_{S(x_0,
t)}Q(x)\,d\mathcal{A}\right)^{\frac{1}{1-n}}, & t\in (0, \delta(x_0)),\\
0, & t\not\in (0, \delta(x_0)).
\end{array}
\right.$$
Заметим теперь, что требование вида~\eqref{eq3AB} выполняется при
$\varepsilon_0=\delta(x_0)$ и всех достаточно малых $\varepsilon.$
Далее установим неравенство
\begin{equation}\label{eq4}
\int\limits_{\varepsilon<d(x, x_0)<\delta(x_0)} Q(x)\psi^n(d(x,
x_0))\,dv(x)\leqslant C\cdot\int\limits_{\varepsilon}^{\delta(x_0)}
\left(\int\limits_{S(x_0,
t)}Q(x)\,d\mathcal{A}\right)^{\frac{1}{1-n}}dt
\end{equation}
при некоторой постоянной $C>0.$ Для этого покажем, что к левой части
соотношения \eqref{eq4} применим аналог теоремы Фубини. Рассмотрим в
окрестности точки $x_0\in S(z_0, r)\subset {\Bbb R}^n$ локальную
систему координат $z^1,\ldots, z^n,$ $n-1$ базисных векторов которой
взаимно ортогональны и лежат в плоскости, касательной к сфере в
точке $x_0,$ а последний базисный вектор перпендикулярен этой
плоскости. Пусть $r, \theta^1,\ldots, \theta^{n-1}$ сферические
координаты точки $x=x(\theta)$ в ${\Bbb R}^n.$ Заметим, что $n-1$
приращений переменных $z^1,\ldots, z^{n-1}$ вдоль сферы при
фиксированном $r$ равны $dz^1=rd\theta^1,\dots,
dz^{n-1}=rd\theta^{n-1},$ а приращение переменной $z^n$ по $r$ равно
$dz^n=dr.$ В таком случае,
$$dv(x)=\sqrt{{\rm det\,}g_{ij}(x)}r^{n-1}\,dr d\theta^1\dots
d\theta^{n-1}.$$
Рассмотрим параметризацию сферы $S(0, r)$ $x=x(\theta),$
$\theta=(\theta^1,\ldots,\theta^{n-1}),$ $\theta_i\in (-\pi, \pi].$
Заметим, что $\frac{\partial x^{\alpha}}{\partial \theta^{\beta}}=1$
при $\alpha=\beta$ и $\frac{\partial x^{\alpha}}{\partial
\theta^{\beta}}=0$ при $\alpha\ne \beta,$ $\alpha,\beta=1,\ldots,
n-1.$ Тогда в обозначениях соотношения~\eqref{eq5} имеем:
$g_{\alpha\beta}^*(\theta)=g_{\alpha\beta}(x(\theta))r^2,$
$$d\mathcal{A}=\sqrt{\det\, g_{\alpha\beta}(x(\theta))}r^{n-1}d{\theta}^1\ldots d{\theta}^{n-1}.$$
Заметим, что
 \begin{equation}\label{eq6}
\int\limits_{S(x_0,
r)}Q(x)\psi^n(d(x,x_0))\,d\mathcal{A}=\psi^n(r)r^{n-1}\cdot\int\limits_{\Pi}\sqrt{{\rm
det\,}g_{\alpha\beta}(x(\theta))}Q(x(\theta))\,d\theta^{1}\dots
d\theta^{n-1},
 \end{equation}
где $\Pi=(-\pi, \pi]^{n-1}$ --- прямоугольная область изменения
параметров $\theta^1,\ldots,\theta^{n-1}.$ Напомним, что в
нормальной системе координат геодезические сферы переходят в обычные
сферы того же радиуса с центром в нуле, а пучок геодезических,
исходящих из точки многообразия, переходит в пучок радиальных
отрезков в ${\Bbb R}^n$ (см.~\cite[леммы~5.9 и 6.11]{Lee}), так что
кольцу $\{x\in {\Bbb M}^n: \varepsilon<d(x, x_0)<\delta(x_0)\}$
соответствует та часть ${\Bbb R}^n,$ в которой $r\in (\varepsilon,
\delta(x_0)).$ Согласно сказанному выше, применяя классическую
теорему Фубини (см., напр.,~\cite[разд.~8.1, гл.~III]{Sa}),
 \begin{multline}\label{eq7}
\int\limits_{\varepsilon<d(x, x_0)<\delta(x_0)} Q(x)\psi^n(d(x,
x_0))\,dv(x)=\\
=\int\limits_{\varepsilon}^{\delta(x_0)}\int\limits_{\Pi} \sqrt{{\rm
det\,}g_{ij}(x)}Q(x)\psi^n(r)r^{n-1}\,d\theta^1\dots
d\theta^{n-1}dr.
 \end{multline}
Поскольку в нормальных координатах тензорная матрица $g_{ij}$  сколь
угодно близка к единичной в окрестности данной точки, то
$C_2\det\,g_{\alpha\beta}(x)\leqslant\det\,g_{ij}(x)\leqslant
C_1\det\,g_{\alpha\beta}(x).$ Учитывая сказанное и сравнивая
\eqref{eq6} и \eqref{eq7}, приходим к соотношению~\eqref{eq4}. Но
тогда также
$$\int\limits_{\varepsilon<d(x, x_0)<\delta(x_0)}
Q(x)\psi^n(d(x, x_0))dv(x)=o(I^n(\varepsilon, \delta(x_0)))$$
ввиду соотношения~\eqref{eq3}. Утверждение теоремы следует теперь из
леммы~\ref{lem1}.~$\Box$
 \end{proof}

\medskip
Ввиду теоремы Арцела--Асколи (см., напр.,~\cite[пункт~20.4]{Va})
имеем также следующее

\medskip
\begin{corollary}\label{cor2} {\sl\, Предположим{\em,} что в условиях
теоремы~{\em\ref{theor4*!}} многообразие ${\Bbb M}_*^n$ является
компактным. Если $\frak{R}_{D, Q, B_R, K}(D)$ --- семейство{\em,}
состоящее из всех открытых дискретных отображений $f\colon
D\rightarrow B_R\setminus K,$ являющихся кольцевыми
$Q$-отоб\-ра\-же\-ни\-я\-ми в каждой точке $x_0\in D$ и{\em,} кроме
того{\em,} условие~\eqref{eq3} выполнено в каждой точке $x_0\in D,$
то класс $\frak{R}_{D, Q, B_R, K}(D)$  образует нормальное семейство
отображений. }
 \end{corollary}

\medskip
\noindent{{\bf Денис Петрович Ильютко} \\
МГУ имени М.\,В.\,Ломоносова \\
кафедра дифференциальной геометрии и приложений, мехмат факультет,\\
Ленинские горы, ГЗ МГУ, ГСП-1, г.~Москва, Россия, 119991\\
тел. +7 495 939 39 40, e-mail: ilyutko@yandex.ru}

\medskip
\noindent{{\bf Евгений Александрович Севостьянов} \\
Житомирский государственный университет им.\ И.~Франко\\
кафедра математического анализа, ул. Большая Бердичевская, 40 \\
г.~Житомир, Украина, 10 008 \\ тел. +38 066 959 50 34 (моб.),
e-mail: esevostyanov2009@mail.ru}

\end{document}